\documentclass[11pt]{article}

\usepackage{amsmath,amsthm}
\usepackage{amssymb}
\usepackage{hyperref}

\allowdisplaybreaks

\begin{document}

\title{Summation formulas of hyperharmonic numbers with their generalizations II}
\author{
Takao Komatsu
\\
\small Department of Mathematical Sciences, School of Science\\
\small Zhejiang Sci-Tech University\\
\small Hangzhou 310018 China\\
\small \texttt{komatsu@zstu.edu.cn}\\\\
Rusen Li\\
\small School of Mathematics\\
\small Shandong University\\
\small Jinan 250100 China\\
\small \texttt{limanjiashe@163.com}
}

\date{
\small MR Subject Classifications: Primary 05A19, 11B65; Secondary 11B73
}

\maketitle

\def\stf#1#2{\left[#1\atop#2\right]}
\def\sts#1#2{\left\{#1\atop#2\right\}}
\def\e{\mathfrak e}
\def\f{\mathfrak f}

\newtheorem{theorem}{Theorem}
\newtheorem{Prop}{Proposition}
\newtheorem{Cor}{Corollary}
\newtheorem{Lem}{Lemma}
\newtheorem{Example}{Example}
\newtheorem{Remark}{Remark}

\begin{abstract}
In 1990, Spie\ss \, gave some identities of harmonic numbers including the types of $\sum_{\ell=1}^n\ell^k H_\ell$, $\sum_{\ell=1}^n\ell^k H_{n-\ell}$ and $\sum_{\ell=1}^n\ell^k H_\ell H_{n-\ell}$. In this paper, we derive several formulas of hyperharmonic numbers including $\sum_{\ell=0}^{n} {\ell}^{p} h_{\ell}^{(r)}  h_{n-\ell}^{(s)}$ and $\sum_{\ell=0}^n \ell^{p}(h_{\ell}^{(r)})^{2}$. Some more formulas of generalized hyperharmonic numbers are also shown.
\end{abstract}

\section{Introduction and preliminaries}

Many researchers have been considering varieties of identities involving harmonic numbers $H_n:=\sum_{j=1}^n 1/j$. One of the most famous types is the so-called Euler sum like
$$
\sum_{k=1}^\infty\frac{(H_k)^m}{(k+1)^n}\quad(m,n\ge 1)\,,
$$
that Euler considered in response to a letter from Goldbach in 1742 (see, e.g., \cite[p.253]{Berndt}). It is interesting that the Riemann zeta functions and their generalizations often appear in such expressions. Borwein-Borwein-Girgensohn theorem \cite{BBG} is useful to compute some kinds of Euler sums.

There are many generalizations of harmonic numbers. One of the most useful ones is the generalized harmonic numbers, defined by
$$
H_n^{(r)}=1+\frac{1}{2^r}+\cdots+\frac{1}{n^r}\,.
$$
Zave \cite{Zave} showed the coefficients of some power series expansions in terms of the generalized harmonic numbers.
One of the most interesting generalized Euler sums is of the form
$$
S_{\pi,q}=\sum_{k=1}^\infty\frac{H_k^{(\pi_1)}H_k^{(\pi_2)}\cdots H_k^{(\pi_r)}}{k^q}
$$
and the quantity $\pi_1+\pi_2+\cdots+\pi_r+q$ is called the weight of the sum.
Note that one part of weights $k^q$ exists in the denominator.

On the contrary, Spie\ss \, \cite{Spiess} gave some identities including the types of $\sum_{\ell=1}^n\ell^k H_\ell$, $\sum_{\ell=1}^n\ell^k H_{n-\ell}$ and $\sum_{\ell=1}^n\ell^k H_\ell H_{n-\ell}$. Here, weights exist in the numerator, that is, they are of the form in the multiplication.

Another kind of generalization of harmonic numbers is the hyperharmonic numbers \cite{BS,BGG,Cereceda,Conway,Dil,Kamano,MS,MD,omur}, defined by
\begin{equation}
h_n^{(r)}=\sum_{\ell=1}^n h_\ell^{(r-1)}\quad\hbox{with}\quad h_n^{(1)}=H_n\,.
\label{def:hyperharmonic}
\end{equation}
In particular, we have
\begin{align}
h_n^{(r)}&=\binom{n+r-1}{r-1}(H_{n+r-1}-H_{r-1})\quad\hbox{\cite{Conway}}
\label{h01}\\
&=\sum_{j=1}^n\binom{n+r-j-1}{r-1}\frac{1}{j}\quad\hbox{\cite{BGG}}\,,
\label{h02}
\end{align}

Let $x$ be an indeterminate. Then following Riordan {\cite{Riordan}}, we write
$$
(x)_{(n)}=x(x-1)\cdots(x-n+1) \quad  (n \in \mathbb N),
$$
for the falling factorial, with $(x)_{(0)}=1$, and $\mathbb N:=\{1,2,3,\cdots\}$.
With this notation, the Stirling numbers of the first kind, denoted by $s(n,k)$, are defined as
$$
(x)_{(n)}=\sum_{k=0}^{n} s(n,k) x^{k} \quad  (n \in \mathbb N).
$$
The Stirling numbers of the second kind, denoted by $S(n,k)$, are defined as
$$
x^{n}=\sum_{k=0}^{n} S(n,k) (x)_{(k)}, \quad  (n \in \mathbb N).
$$
These numbers may also be defined recursively by
\begin{align*}
s(n+1,k)&=s(n,k-1)-n \cdot s(n,k),\\
S(n+1,k)&=S(n,k-1)+k \cdot S(n,k)
\end{align*}
with boundary values
$$
s(n,0)=s(0,n)=S(n,0)=S(0,n)=\delta_{n 0} \quad  (n \geq 0),
$$
where $\delta_{n m}$ is the Kronecker delta, that is, $\delta_{n n}=1$, $\delta_{n m}=0$ for $n \neq m$. Tables of the Stirling numbers can be found in Abramowitz and Stegun {\cite{Abra}}.
The unsigned Stirling numbers of the first kind, denoted by $s_{u}(n,k)$, are defined as
$$
(x)^{(n)}=\sum_{k=0}^{n} s_{u}(n,k) x^{k} \quad  (n \in \mathbb N)\,.
$$
where $(x)^{(n)}=x(x+1)\cdots(x+n-1)$ ($n\ge 1$) denotes the rising factorial with $(x)^{(0)}=1$.
It is well-known that $s_{u}(n,k)=(-1)^{n+k}s(n,k)$.

The generalized hyperharmonic numbers are defined by (see {\cite{Dil,Rusen,omur}})
$$
H_n^{(p,r)}:=\sum_{j=1}^n H_{j}^{(p,r-1)} \quad (n,p,r \in \mathbb N)\,,
$$
with
$$
H_n^{(p,1)}=H_n^{(p)}\,.
$$
Observing $H_n^{(1,r)}=h_n^{(r)}$, we see that the generalized hyperharmonic numbers are unified extensions of both generalized harmonic numbers and hyperharmonic numbers.

According to the spirit of Spie\ss, the authors \cite{Komatsu} derived several formulas of hyperharmonic numbers of type $\sum_{\ell=1}^{n} {\ell}^{p} H_{\ell}^{(r)}$ and $\sum_{\ell=0}^{n} {\ell}^{p} H_{n-\ell}^{(r)}$. Several formulas of $q$-hyperharmonic numbers are also derived as $q$-generalizations.
The purpose of this paper is to show several identities$$
\sum_{\ell=0}^{n} {\ell}^{p} h_{\ell}^{(r)}  h_{n-\ell}^{(s)},\quad
\sum_{\ell=0}^n (\ell)^{(p)}h_{\ell}^{(r)}  h_{n-\ell}^{(s)},\quad
\sum_{\ell=0}^n \ell^{p}(h_{\ell}^{(r)})^{2},\quad
\sum_{\ell=0}^n (\ell)^{(p)}(h_{\ell}^{(r)})^{2}
$$
and
$$
\sum_{\ell=0}^n \ell^{p}H_{\ell}^{(q,r)},\quad
\sum_{\ell=0}^n \ell^{p}H_{n-\ell}^{(r)},\quad
\sum_{\ell=0}^n \ell^{p}H_{n-\ell}^{(q,r)}.
$$

This paper is also motivated from the summation $\sum_{\ell=1}^n\ell^k$, which is related to Bernoulli numbers (e.g., \cite{Carlitz,CFZ}), because we can think of harmonic numbers or hyperharmonic numbers as the weight of the summation $\sum_{\ell=1}^n\ell^k$.

\section{Some fundamental results and lemmata}

\subsection{Some fundamental results}

Spie\ss \, \cite{Spiess} gives some identities including the types of $\sum_{\ell=0}^n\ell^k H_\ell$ and $\sum_{\ell=0}^n\ell^k H_{n-\ell}$ in terms of $H_{n+1}$. More precisely, it is shown {\cite{Spiess}} that for $n, k, r \in \mathbb N$,
\begin{align*}
&\sum_{\ell=0}^{n}{\ell}^{k}=A(k,n), \notag\\
&\sum_{\ell=0}^{n} {\ell}^{k} H_{\ell}=A(k,n)H_{n+1}-B(k,n), \notag\\
&\sum_{\ell=0}^{n} {\ell}^{k} H_{n-\ell}=A(k,n)H_{n+1}-C(k,n), \notag\\
&\sum_{\ell=0}^{n} {\ell}^{k} H_{\ell}H_{n-\ell}=A(k,n)S_{n+1}-(B(k,n)+C(k,n))H_{n+1}+D(k,n), \notag\\
&\sum_{\ell=0}^{n}{\ell}^{k} H_{\ell}^{(2)}=A(k,n)(H_{n})^{2}-(-1)^{p}B_{p}^{+}H_{n}+E(k,n), \notag\\
&\sum_{\ell=0}^{n}{\ell}^{k} (H_{\ell})^{2}=A(k,n)(H_{n})^{2}-F(k,n)H_{n}+G(k,n), \notag\\
&\sum_{\ell=0}^{n} {\ell}^{k} H_{\ell}^{(r)}=A(k,n) H_{n}^{(r)}-\sum_{t=1}^{k+1} H_{n}^{(r-t)} H(k,t)\,,\notag
\end{align*}
where $A(k,n)$, $B(k,n)$, $C(k,n)$, $D(k,n)$, $E(k,n)$, $F(k,n)$ and $G(k,n)$ are polynomials in $n$,
$S_n=(H_n)^2-H_n^{(2)}$, $B_n^{+}$ is the well-known Bernoulli numbers.
The Bernoulli numbers $B_n^{+}$ are determined by the recurrence formula
$$
\sum_{j=0}^k\binom{k+1}{j}B_j^{+}=k+1\quad (k\ge 0)
$$
or by the generating function
\begin{align*}
\frac{t}{1-e^{-t}}=\sum_{n=0}^\infty B_n^{+}\frac{t^n}{n!}\,.
\end{align*}

The $A(k,n)$ are well-known (see, e.g. Riordan {\cite{Riordan}}). Spie\ss\, {\cite{Spiess}} gives explicit forms for $k=0,1,2,3$. More generally, it is shown that for $n, k \in \mathbb N$,
\begin{align*}
&A(k,n)=\sum_{\ell=0}^{k} S(k,\ell) {\ell!} \binom{n+1}{\ell+1},\\
&B(k,n)=\sum_{\ell=0}^{k} \frac{1}{\ell+1} S(k,\ell) {\ell!} \binom{n+1}{\ell+1},\\
&C(k,n)=\sum_{\ell=0}^{k} S(k,\ell) {\ell!} \binom{n+1}{\ell+1} H_{\ell+1},\\
&D(k,n)=\sum_{\ell=0}^{k} S(k,\ell) {\ell!} \binom{n+1}{\ell+1} \biggl(\frac{1}{\ell+1}H_{\ell+1}+H_{\ell+1}^{2}\biggl)\,,\\
&E(k,n)=\sum_{\ell=0}^{k} S(k,\ell) {\ell!} A_{\ell}(n)\,,\\
&F(k,n)=\sum_{\ell=0}^{k} S(k,\ell) {\ell!} \frac{1}{\ell+1} \biggl(2\binom{n}{\ell+1}+(-1)^{\ell}\biggr)\,,\\
&G(k,n)=\sum_{\ell=0}^{k} S(k,\ell) {\ell!} \biggl(\frac{2}{(\ell+1)^{2}}\binom{n}{\ell+1}-\frac{1}{(\ell+1)!}\sum_{k=2}^{\ell+1}s(\ell+1,k)H_{n}^{(2-\ell)}\biggr)\,,\\
&H(k,t)=\sum_{\ell=t-1}^{k} \frac{1}{\ell+1} S(k,\ell) s(\ell+1,t)\,,
\end{align*}
where $A_{\ell}(n)$ is a polynomial in $n$ of degree $\ell$ (see {\cite{Spiess}}).

For a positive integer $n$, we have
\begin{align}
\sum_{\ell=1}^n\ell H_{\ell}H_{n-\ell}=\frac{n(n+1)}{2}S_n-n^2 H_n+n^2\,,\quad {\cite{Spiess}} \label{bhh6}
\end{align}
where
$$
S_n:=(H_n)^2-H_n^{(2)}=\frac{2}{n!}s_{u}(n+1,3)=\sum_{k=1}^n\frac{H_{n-k}}{k}\,.
$$
Note that
$$
\sum_{k=1}^n\frac{H_{k}}{k}=\frac{(H_n)^2+H_n^{(2)}}{2}\,.\quad {\cite{Spiess}}
$$

We see that
\begin{multline}
\sum_{\ell=1}^n\ell^2 H_{\ell}H_{n-\ell}=\frac{n(n+1)(2 n+1)}{6}S_n\\
-\frac{n(13 n^2+6 n-1)}{18}H_n+\frac{n(71 n^2+30 n+7)}{108}\,,\quad {\cite{Spiess}} \label{bhh4}
\end{multline}
\begin{multline}
\sum_{\ell=1}^n\ell^3 H_{\ell}H_{n-\ell}=\left(\frac{n(n+1)}{2}\right)^{2}S_n\\
-\frac{n^2(n+1)(7 n-1)}{12}H_n+\frac{n^2(35 n^2+30 n+7)}{72}\,.\quad {\cite{Spiess}} \label{bhh5}
\end{multline}
Combining the identities (\ref{bhh6}), (\ref{bhh4}), (\ref{bhh5}), we have
\begin{multline*}
\sum_{\ell=1}^n\ell H_{\ell}^{(2)}H_{n-\ell}^{(2)}=\frac{n(n+1)(n+2)(n+3)}{12}S_n\\
-\frac{n^2(11 n^2+48 n+49)}{36}H_n+\frac{n^2(85 n^2+312 n+251)}{216}\,.
\end{multline*}

\subsection{Some lemmata}

Spie\ss \, \cite{Spiess} gives the following lemma, which will be frequently used in later sections.
\begin{Lem}[\cite{Spiess}]
Given summation formulas $\sum_{\ell=0}^{n} \binom{\ell}{p}c_{\ell}=F(n,p), n, p \in \mathbb N$, one has
$$
\sum_{\ell=0}^{n} {\ell}^{p} c_{\ell}=\sum_{\ell=0}^{p} S(p,\ell) \cdot {\ell}! \cdot F(n,\ell)\,.
\label{spiesslemma}
$$
\end{Lem}

\begin{Lem}[{\cite[Lemma 2.4]{Kamano}}]
The following identities hold:
\begin{align}\label{Kamano1}
\binom{n+r-1}{r-1}=\frac{1}{(r-1)!}\sum_{k=1}^{r} s_{u}(r,k) n^{k-1}\,.
\end{align}
\begin{align}\label{Kamano2}
\binom{n+r-1}{r-1}\biggl(\frac{1}{n+1}+\cdots+\frac{1}{n+r-1}\biggr)=\frac{1}{(r-1)!}\sum_{k=1}^{r-1} s_{u}(r,k+1) k n^{k-1}\,.
\end{align}
\end{Lem}

We now prove some lemmas which will be needed in later sections.

\begin{Lem}\label{binom1}
For $r, s, n \in \mathbb N$ with $0 \le \ell \le n$, we have
\begin{align*}
\binom{\ell+r-1}{r-1}\binom{n-\ell+s-1}{s-1}=\sum_{t=0}^{r+s-2}a(r,s,n,t)\ell^t\,,
\end{align*}
where
\begin{multline*}
a(r,s,n,t)=\frac{1}{(r-1)!(s-1)!}\\
\times\sum_{\substack{t_{1}+t_{2}=t\\ 0 \le t_{1} \le r-1\\ 0 \le t_{2} \le s-1 }} s_{u}(r,t_{1}+1) \sum_{k_{2}=t_{2}}^{s-1} s_{u}(s,k_{2}+1) \binom{k_{2}}{t_{2}} (-1)^{t_{2}} n^{k_{2}-t_{2}}.
\end{multline*}
\end{Lem}
\begin{proof}
By using \ref{Kamano1}, we have
\begin{align*}
&\quad \binom{\ell+r-1}{r-1}\binom{n-\ell+s-1}{s-1}\\
&=\frac{1}{(r-1)!}\sum_{k_{1}=1}^{r} s_{u}(r,k_{1}) \ell^{k_{1}-1}\frac{1}{(s-1)!}\sum_{k_{2}=1}^{s} s_{u}(s,k_{2}) (n-\ell)^{k_{2}-1}\\
&=\frac{1}{(r-1)!(s-1)!}\\
&\quad\times\sum_{t_{1}=0}^{r-1} s_{u}(r,t_{1}+1) \ell^{t_{1}}\sum_{k_{2}=0}^{s-1} s_{u}(s,k_{2}+1) \sum_{t_{2}=0}^{k_{2}}\binom{k_{2}}{t_{2}}\ell^{t_{2}}(-1)^{t_{2}}n^{k_{2}-t_{2}}\\
&=\frac{1}{(r-1)!(s-1)!}\\
&\quad\times\sum_{t_{1}=0}^{r-1} s_{u}(r,t_{1}+1) \ell^{t_{1}} \sum_{t_{2}=0}^{s-1}\ell^{t_{2}}\sum_{k_{2}=t_{2}}^{s-1}s_{u}(s,k_{2}+1)\binom{k_{2}}{t_{2}}(-1)^{t_{2}}n^{k_{2}-t_{2}}\\
&=\frac{1}{(r-1)!(s-1)!}\\
&\quad\times\sum_{t=0}^{r+s-2}\ell^t \sum_{\substack{t_{1}+t_{2}=t\\ 0 \le t_{1} \le r-1\\ 0 \le t_{2} \le s-1 }} s_{u}(r,t_{1}+1) \sum_{k_{2}=t_{2}}^{s-1} s_{u}(s,k_{2}+1) \binom{k_{2}}{t_{2}} (-1)^{t_{2}} n^{k_{2}-t_{2}}\,.
\end{align*}
\end{proof}

\begin{Lem}\label{binom2}
For $r, s, n \in \mathbb N$ with $0 \le \ell \le n$, we have
\begin{align*}
&\binom{\ell+r-1}{r-1}\binom{n-\ell+s-1}{s-1}\biggl(\frac{1}{n-\ell+1}+\cdots+\frac{1}{n-\ell+s-1}\biggr)\\
&=\sum_{t=0}^{r+s-3}a_{1}(r,s,n,t)\ell^t\,,
\end{align*}
where
\begin{multline*}
\quad a_{1}(r,s,n,t)=\frac{1}{(r-1)!(s-1)!}\\
\times \sum_{\substack{t_{1}+t_{2}=t\\ 0 \le t_{1} \le r-1\\ 0 \le t_{2} \le s-2 }} s_{u}(r,t_{1}+1) \sum_{k_{2}=t_{2}}^{s-2} s_{u}(s,k_{2}+2) (k_{2}+1) \binom{k_{2}}{t_{2}} (-1)^{t_{2}} n^{k_{2}-t_{2}}\,.
\end{multline*}
\end{Lem}
\begin{proof}
By using (\ref{Kamano1}) and (\ref{Kamano2}), we have
\begin{align*}
&\quad \binom{\ell+r-1}{r-1}\binom{n-\ell+s-1}{s-1}\biggl(\frac{1}{n-\ell+1}+\cdots+\frac{1}{n-\ell+s-1}\biggr)\\
&=\frac{1}{(r-1)!}\sum_{k_{1}=1}^{r} s_{u}(r,k_{1}) \ell^{k_{1}-1}\frac{1}{(s-1)!}\sum_{k_{2}=1}^{s-1} s_{u}(s,k_{2}+1) k_{2} (n-\ell)^{k_{2}-1}\\
&=\frac{1}{(r-1)!(s-1)!}\sum_{t_{1}=0}^{r-1} s_{u}(r,t_{1}+1) \ell^{t_{1}}\\
&\quad\times\sum_{k_{2}=0}^{s-2} s_{u}(s,k_{2}+2) (k_{2}+1)\sum_{t_{2}=0}^{k_{2}}\binom{k_{2}}{t_{2}}\ell^{t_{2}}(-1)^{t_{2}}n^{k_{2}-t_{2}}\\
&=\frac{1}{(r-1)!(s-1)!}\sum_{t_{1}=0}^{r-1} s_{u}(r,t_{1}+1) \ell^{t_{1}}\\
&\quad\times \sum_{t_{2}=0}^{s-2}\ell^{t_{2}}\sum_{k_{2}=t_{2}}^{s-2}s_{u}(s,k_{2}+2) (k_{2}+1)\binom{k_{2}}{t_{2}}(-1)^{t_{2}}n^{k_{2}-t_{2}}\\
&=\frac{1}{(r-1)!(s-1)!}\sum_{t=0}^{r+s-3}\ell^t \sum_{\substack{t_{1}+t_{2}=t\\ 0 \le t_{1} \le r-1\\ 0 \le t_{2} \le s-2 }} s_{u}(r,t_{1}+1)\\
&\quad\times \sum_{k_{2}=t_{2}}^{s-2} s_{u}(s,k_{2}+1)(k_{2}+1) \binom{k_{2}}{t_{2}} (-1)^{t_{2}} n^{k_{2}-t_{2}}\,.
\end{align*}
\end{proof}

\begin{Lem}\label{binom3}
For $r, s, n \in \mathbb N$ with $0 \le \ell \le n$, we have
\begin{align*}
\binom{\ell+r-1}{r-1}\binom{n-\ell+s-1}{s-1}\biggl(\frac{1}{\ell+1}+\cdots+\frac{1}{\ell+r-1}\biggr)
=\sum_{t=0}^{r+s-3}a_{2}(r,s,n,t)\ell^t\,,
\end{align*}
where
\begin{multline*}
a_{2}(r,s,n,t)=\frac{1}{(r-1)!(s-1)!}\\
\times\sum_{\substack{t_{1}+t_{2}=t\\ 0 \le t_{1} \le r-2\\ 0 \le t_{2} \le s-1 }} s_{u}(r,t_{1}+2)(t_{1}+1) \sum_{k_{2}=t_{2}}^{s-1} s_{u}(s,k_{2}+1) \binom{k_{2}}{t_{2}} (-1)^{t_{2}} n^{k_{2}-t_{2}}\,.
\end{multline*}
\end{Lem}
\begin{proof}
By using (\ref{Kamano1}) and (\ref{Kamano2}), we have
\begin{align*}
&\quad \binom{\ell+r-1}{r-1}\binom{n-\ell+s-1}{s-1}\biggl(\frac{1}{\ell+1}+\cdots+\frac{1}{\ell+r-1}\biggr)\\
&=\frac{1}{(r-1)!}\sum_{k_{1}=1}^{r-1} s_{u}(r,k_{1}+1)k_{1} \ell^{k_{1}-1}\frac{1}{(s-1)!}\sum_{k_{2}=1}^{s} s_{u}(s,k_{2}) (n-\ell)^{k_{2}-1}\\
&=\frac{1}{(r-1)!(s-1)!}\sum_{t_{1}=0}^{r-2} s_{u}(r,t_{1}+2)(t_{1}+1) \ell^{t_{1}}\\
&\quad\times\sum_{k_{2}=0}^{s-1} s_{u}(s,k_{2}+1) \sum_{t_{2}=0}^{k_{2}}\binom{k_{2}}{t_{2}}\ell^{t_{2}}(-1)^{t_{2}}n^{k_{2}-t_{2}}\\
&=\frac{1}{(r-1)!(s-1)!}\sum_{t_{1}=0}^{r-2} s_{u}(r,t_{1}+2)(t_{1}+1) \ell^{t_{1}}\\
&\quad\times \sum_{t_{2}=0}^{s-1}\ell^{t_{2}}\sum_{k_{2}=t_{2}}^{s-1}s_{u}(s,k_{2}+1)\binom{k_{2}}{t_{2}}(-1)^{t_{2}}n^{k_{2}-t_{2}}\\
&=\frac{1}{(r-1)!(s-1)!}\sum_{t=0}^{r+s-3}\ell^t \sum_{\substack{t_{1}+t_{2}=t\\ 0 \le t_{1} \le r-1\\ 0 \le t_{2} \le s-1 }} s_{u}(r,t_{1}+2)(t_{1}+1)\\
&\quad\times \sum_{k_{2}=t_{2}}^{s-1} s_{u}(s,k_{2}+1) \binom{k_{2}}{t_{2}} (-1)^{t_{2}} n^{k_{2}-t_{2}}\,.
\end{align*}
\end{proof}

\begin{Lem}\label{binom4}
For $r, s, n \in \mathbb N$ with $0 \le \ell \le n$, we have
\begin{multline*}
\binom{\ell+r-1}{r-1}\binom{n-\ell+s-1}{s-1}
\biggl(\frac{1}{\ell+1}+\cdots+\frac{1}{\ell+r-1}\biggr)\\
\times \biggl(\frac{1}{n-\ell+1}+\cdots+\frac{1}{n-\ell+s-1}\biggr)
=\sum_{t=0}^{r+s-4}a_{3}(r,s,n,t)\ell^t\,,
\end{multline*}
where
\begin{multline*}
\quad a_{3}(r,s,n,t)=\frac{1}{(r-1)!(s-1)!}\sum_{\substack{t_{1}+t_{2}=t\\ 0 \le t_{1} \le r-2\\ 0 \le t_{2} \le s-2 }} s_{u}(r,t_{1}+2)(t_{1}+1)\\
\times \sum_{k_{2}=t_{2}}^{s-1} s_{u}(s,k_{2}+2)(k_{2}+1) \binom{k_{2}}{t_{2}} (-1)^{t_{2}} n^{k_{2}-t_{2}}\,.
\end{multline*}
\end{Lem}
\begin{proof}
By using (\ref{Kamano1}) and (\ref{Kamano2}), we have
\begin{align*}
&\binom{\ell+r-1}{r-1}\binom{n-\ell+s-1}{s-1}
\biggl(\frac{1}{\ell+1}+\cdots+\frac{1}{\ell+r-1}\biggr)\\
&\qquad\times \biggl(\frac{1}{n-\ell+1}+\cdots+\frac{1}{n-\ell+s-1}\biggr)\\
&=\frac{1}{(r-1)!}\sum_{k_{1}=1}^{r-1} s_{u}(r,k_{1}+1)k_{1} \ell^{k_{1}-1}\frac{1}{(s-1)!}\sum_{k_{2}=1}^{s-1} s_{u}(s,k_{2}+1) k_{2}(n-\ell)^{k_{2}-1}\\
&=\frac{1}{(r-1)!(s-1)!}\sum_{t_{1}=0}^{r-2} s_{u}(r,t_{1}+2)(t_{1}+1) \ell^{t_{1}}\\
&\qquad\times\sum_{k_{2}=0}^{s-2} s_{u}(s,k_{2}+2) (k_{2}+1) \sum_{t_{2}=0}^{k_{2}}\binom{k_{2}}{t_{2}}\ell^{t_{2}}(-1)^{t_{2}}n^{k_{2}-t_{2}}\\
&=\frac{1}{(r-1)!(s-1)!}\sum_{t_{1}=0}^{r-2} s_{u}(r,t_{1}+2)(t_{1}+1) \ell^{t_{1}}\\
&\qquad\times \sum_{t_{2}=0}^{s-2}\ell^{t_{2}}\sum_{k_{2}=t_{2}}^{s-2}s_{u}(s,k_{2}+2)(k_{2}+1)\binom{k_{2}}{t_{2}}(-1)^{t_{2}}n^{k_{2}-t_{2}}\\
&=\frac{1}{(r-1)!(s-1)!}\sum_{t=0}^{r+s-4}\ell^t \sum_{\substack{t_{1}+t_{2}=t\\ 0 \le t_{1} \le r-2\\ 0 \le t_{2} \le s-2 }} s_{u}(r,t_{1}+2)(t_{1}+1)\\
&\qquad\times\sum_{k_{2}=t_{2}}^{s-2} s_{u}(s,k_{2}+2)(k_{2}+1) \binom{k_{2}}{t_{2}} (-1)^{t_{2}} n^{k_{2}-t_{2}}\,.
\end{align*}
\end{proof}

\section{Main results}

In this section, we study the summations
$$
\sum_{\ell=0}^{n} {\ell}^{p} h_{\ell}^{(r)}  h_{n-\ell}^{(s)},\quad
\sum_{\ell=0}^n (\ell)^{(p)}h_{\ell}^{(r)}  h_{n-\ell}^{(s)},\quad
\sum_{\ell=0}^n \ell^{p}(h_{\ell}^{(r)})^{2},\quad
\sum_{\ell=0}^n (\ell)^{(p)}(h_{\ell}^{(r)})^{2}\,.
$$

\begin{theorem}\label{hypersturc2}
For $n,r,p\ge 1$, we have
\begin{align}
\sum_{\ell=0}^{n} {\ell}^{p} h_{\ell}^{(r)}  h_{n-\ell}^{(s)}=A(p,r,s,n)S_{n}+B(p,r,s,n)H_{n}+C(p,r,s,n)\,,
\end{align}
where \begin{align*}
&A(p,r,s,n)=\sum_{t=0}^{r+s-2} a(r,s,n,t)A(p+t,n),\\
&B(p,r,s,n)=\sum_{t=0}^{r+s-3} (a_{1}(r,s,n,t)+a_{2}(r,s,n,t))A(p+t,n)\\
&\quad+\sum_{t=0}^{r+s-2} a(r,s,n,t)\\
&\qquad\times\biggl(A(p+t,n)(\frac{2}{n+1}-H_{s-1}-H_{r-1})-B(p+t,n)-C(p+t,n)\biggr)\,,\\
&C(p,r,s,n)=\sum_{t=0}^{r+s-2} a(r,s,n,t)\biggl(D(p+t,n)-\frac{1}{n+1}(B(p+t,n)-C(p+t,n))\biggr)\\
&+\sum_{t=0}^{r+s-3} a_{1}(r,s,n,t)\biggl(A(p+t,n)\frac{1}{n+1}-B(p+t,n)\biggr)\\
&-H_{s-1}\sum_{t=0}^{r+s-2} a_{1}(r,s,n,t)\biggl(A(p+t,n)\frac{1}{n+1}-B(p+t,n)\biggr)\\
&+\sum_{t=0}^{r+s-3} a_{2}(r,s,n,t)\biggl(A(p+t,n)\frac{1}{n+1}-C(p+t,n)\biggr)\\
&+\sum_{t=0}^{r+s-4} a_{3}(r,s,n,t) \sum_{\ell=0}^{n} {\ell}^{p+t}\\
&-H_{s-1}\sum_{t=0}^{r+s-3} a_{2}(r,s,n,t) \sum_{\ell=0}^{n} {\ell}^{p+t}\\
&-H_{r-1}\sum_{t=0}^{r+s-2} a(r,s,n,t)\biggl(A(p+t,n)\frac{1}{n+1}-C(p+t,n)\biggr)\\
&-H_{r-1}\sum_{t=0}^{r+s-3} a_{1}(r,s,n,t) \sum_{\ell=0}^{n} {\ell}^{p+t}+H_{r-1}H_{s-1}\sum_{t=0}^{r+s-2} a(r,s,n,t) \sum_{\ell=0}^{n} {\ell}^{p+t}\,.
\end{align*}

\end{theorem}
\begin{proof}
With the help of (\ref{h01}), we have
\begin{align*}
&\sum_{\ell=0}^{n} {\ell}^{p} h_{\ell}^{(r)}  h_{n-\ell}^{(s)}\notag\\
&=\sum_{\ell=0}^{n} {\ell}^{p} \binom{\ell+r-1}{r-1}(H_{\ell+r-1}-H_{r-1}) \binom{n-\ell+s-1}{s-1}(H_{n-\ell+s-1}-H_{s-1})\notag\\
&=\sum_{\ell=0}^{n} {\ell}^{p} \binom{\ell+r-1}{r-1}\binom{n-\ell+s-1}{s-1}\\
&\quad\times\biggl(H_{\ell}H_{n-\ell}+H_{\ell}\biggl(\frac{1}{n-\ell+1}+\cdots+\frac{1}{n-\ell+s-1}\biggr)\\
&\qquad -H_{\ell}H_{s-1}+\biggl(\frac{1}{\ell+1}+\cdots+\frac{1}{\ell+r-1}\biggr)H_{n-\ell}\\
&\qquad +\biggl(\frac{1}{\ell+1}+\cdots+\frac{1}{\ell+r-1}\biggr)\biggl(\frac{1}{n-\ell+1}+\cdots+\frac{1}{n-\ell+s-1}\biggr)\\
&\qquad -\biggl(\frac{1}{\ell+1}+\cdots+\frac{1}{\ell+r-1}\biggr)H_{s-1}-H_{r-1}H_{n-\ell}\\
&\qquad -H_{r-1}\biggl(\frac{1}{n-\ell+1}+\cdots+\frac{1}{n-\ell+s-1}\biggr)+H_{r-1}H_{s-1}
\biggr)\,.
\end{align*}
The above sum can be divided into $9$ parts. With the help of Lemmas \ref{binom1},  \ref{binom2}, \ref{binom3} and \ref{binom4}, we have the following $9$ identities:
\begin{align*}
&{\rm Part ~I}\\
&=\sum_{\ell=0}^{n} {\ell}^{p} \binom{\ell+r-1}{r-1}\binom{n-\ell+s-1}{s-1}H_{\ell}H_{n-\ell}\\
&=\sum_{\ell=0}^{n} {\ell}^{p} H_{\ell}H_{n-\ell} \sum_{t=0}^{r+s-2}a(r,s,n,t)\ell^t\\
&=\sum_{t=0}^{r+s-2} a(r,s,n,t) \sum_{\ell=0}^{n} {\ell}^{p+t} H_{\ell}H_{n-\ell}\\
&=\sum_{t=0}^{r+s-2} a(r,s,n,t)\\
&\qquad\times\biggl( A(p+t,n)S_{n+1}-(B(p+t,n)+C(p+t,n))H_{n+1}+D(p+t,n)\biggr)\\
&=\sum_{t=0}^{r+s-2} a(r,s,n,t)\\
&\qquad\times\biggl( A(p+t,n)S_{n}+\biggl(2 A(p+t,n)\frac{1}{n+1}-B(p+t,n)-C(p+t,n)\biggr)H_{n}\\
&\quad +D(p+t,n)-\frac{1}{n+1}(B(p+t,n)-C(p+t,n))\biggr)\\
&=S_{n} \sum_{t=0}^{r+s-2} a(r,s,n,t)A(p+t,n)\\
&\quad +H_{n}\sum_{t=0}^{r+s-2} a(r,s,n,t)\biggl(2 A(p+t,n)\frac{1}{n+1}-B(p+t,n)-C(p+t,n)\biggr)\\
&\quad +\sum_{t=0}^{r+s-2} a(r,s,n,t)\biggl(D(p+t,n)-\frac{1}{n+1}(B(p+t,n)-C(p+t,n))\biggr)\,.
\end{align*}

\begin{align*}
&{\rm Part ~II}\\
&=\sum_{\ell=0}^{n} {\ell}^{p} \binom{\ell+r-1}{r-1}\binom{n-\ell+s-1}{s-1}H_{\ell}\biggl(\frac{1}{n-\ell+1}+\cdots+\frac{1}{n-\ell+s-1}\biggr)\\
&=\sum_{\ell=0}^{n} {\ell}^{p} H_{\ell} \sum_{t=0}^{r+s-3}a_{1}(r,s,n,t)\ell^t\\
&=\sum_{t=0}^{r+s-3} a_{1}(r,s,n,t) \sum_{\ell=0}^{n} {\ell}^{p+t} H_{\ell}\\
&=\sum_{t=0}^{r+s-3} a_{1}(r,s,n,t)(A(p+t,n)H_{n+1}-B(p+t,n))\\
&=H_{n}\sum_{t=0}^{r+s-3} a_{1}(r,s,n,t)A(p+t,n)\\
&\qquad +\sum_{t=0}^{r+s-3} a_{1}(r,s,n,t)\biggl(A(p+t,n)\frac{1}{n+1}-B(p+t,n)\biggr)\,.
\end{align*}

\begin{align*}
&{\rm Part ~III}\\
&=-\sum_{\ell=0}^{n} {\ell}^{p} \binom{\ell+r-1}{r-1}\binom{n-\ell+s-1}{s-1}H_{\ell}H_{s-1}\\
&=-\sum_{\ell=0}^{n} {\ell}^{p} H_{\ell}H_{s-1} \sum_{t=0}^{r+s-2}a(r,s,n,t)\ell^t\\
&=-H_{s-1}\sum_{t=0}^{r+s-2} a(r,s,n,t) \sum_{\ell=0}^{n} {\ell}^{p+t} H_{\ell}\\
&=-H_{n}H_{s-1}\sum_{t=0}^{r+s-2} a(r,s,n,t)A(p+t,n)\\
&\quad -H_{s-1}\sum_{t=0}^{r+s-2} a_{1}(r,s,n,t)\biggl(A(p+t,n)\frac{1}{n+1}-B(p+t,n)\biggr)\,.
\end{align*}

\begin{align*}
&{\rm Part ~IV}\\
&=\sum_{\ell=0}^{n} {\ell}^{p} \binom{\ell+r-1}{r-1}\binom{n-\ell+s-1}{s-1}\biggl(\frac{1}{\ell+1}+\cdots+\frac{1}{\ell+r-1}\biggr)H_{n-\ell}\\
&=\sum_{\ell=0}^{n} {\ell}^{p} H_{n-\ell} \sum_{t=0}^{r+s-3}a_{2}(r,s,n,t)\ell^t\\
&=\sum_{t=0}^{r+s-3} a_{2}(r,s,n,t) \sum_{\ell=0}^{n} {\ell}^{p+t} H_{n-\ell}\\
&=\sum_{t=0}^{r+s-3} a_{2}(r,s,n,t)(A(p+t,n)H_{n+1}-C(p+t,n))\\
&=H_{n}\sum_{t=0}^{r+s-3} a_{2}(r,s,n,t)A(p+t,n)\\
&\quad +\sum_{t=0}^{r+s-3} a_{2}(r,s,n,t)\biggl(A(p+t,n)\frac{1}{n+1}-C(p+t,n)\biggr)\,.
\end{align*}

\begin{align*}
&{\rm Part ~V}\\
&=\sum_{\ell=0}^{n} {\ell}^{p} \binom{\ell+r-1}{r-1}\binom{n-\ell+s-1}{s-1}\\
&\quad \times \biggl(\frac{1}{\ell+1}+\cdots+\frac{1}{\ell+r-1}\biggr)
\biggl(\frac{1}{n-\ell+1}+\cdots+\frac{1}{n-\ell+s-1}\biggr)\\
&=\sum_{\ell=0}^{n} {\ell}^{p} \sum_{t=0}^{r+s-4}a_{3}(r,s,n,t)\ell^t\\
&=\sum_{t=0}^{r+s-4} a_{3}(r,s,n,t) \sum_{\ell=0}^{n} {\ell}^{p+t}\,.
\end{align*}

\begin{align*}
&{\rm Part ~VI}\\
&=-\sum_{\ell=0}^{n} {\ell}^{p} \binom{\ell+r-1}{r-1}\binom{n-\ell+s-1}{s-1}\biggl(\frac{1}{\ell+1}+\cdots+\frac{1}{\ell+r-1}\biggr)H_{s-1}\\
&=-H_{s-1}\sum_{\ell=0}^{n} {\ell}^{p} \sum_{t=0}^{r+s-3}a_{2}(r,s,n,t)\ell^t\\
&=-H_{s-1}\sum_{t=0}^{r+s-3} a_{2}(r,s,n,t) \sum_{\ell=0}^{n} {\ell}^{p+t}\,.
\end{align*}

\begin{align*}
&{\rm Part ~VII}\\
&=-\sum_{\ell=0}^{n} {\ell}^{p} \binom{\ell+r-1}{r-1}\binom{n-\ell+s-1}{s-1}H_{r-1}H_{n-\ell}\\
&=-H_{r-1}\sum_{\ell=0}^{n} {\ell}^{p}H_{n-\ell} \sum_{t=0}^{r+s-2}a(r,s,n,t)\ell^t \notag\\
&=-H_{r-1}\sum_{t=0}^{r+s-2} a(r,s,n,t)(A(p+t,n)H_{n+1}-C(p+t,n))\\
&=-H_{n}H_{r-1}\sum_{t=0}^{r+s-2} a(r,s,n,t)A(p+t,n)\\
&\quad -H_{r-1}\sum_{t=0}^{r+s-2} a(r,s,n,t)\biggl(A(p+t,n)\frac{1}{n+1}-C(p+t,n)\biggr)\,.
\end{align*}

\begin{align*}
&{\rm Part ~VIII}\\
&=-\sum_{\ell=0}^{n} {\ell}^{p} \binom{\ell+r-1}{r-1}\binom{n-\ell+s-1}{s-1}H_{r-1}\biggl(\frac{1}{n-\ell+1}+\cdots+\frac{1}{n-\ell+s-1}\biggr)\\
&=-H_{r-1}\sum_{\ell=0}^{n} {\ell}^{p}\sum_{t=0}^{r+s-3}a_{1}(r,s,n,t)\ell^t\\
&=-H_{r-1}\sum_{t=0}^{r+s-3} a_{1}(r,s,n,t) \sum_{\ell=0}^{n} {\ell}^{p+t}\,.
\end{align*}

\begin{align*}
&{\rm Part ~IX}\\
&=\sum_{\ell=0}^{n} {\ell}^{p} \binom{\ell+r-1}{r-1}\binom{n-\ell+s-1}{s-1}H_{r-1}H_{s-1} \\
&=H_{r-1}H_{s-1}\sum_{\ell=0}^{n} {\ell}^{p}\sum_{t=0}^{r+s-2}a(r,s,n,t)\ell^t\\
&=H_{r-1}H_{s-1}\sum_{t=0}^{r+s-2} a(r,s,n,t) \sum_{\ell=0}^{n} {\ell}^{p+t}\,.
\end{align*}

Combining all these $9$ parts together, we get the desired result.
\end{proof}

\begin{theorem}\label{hypersturc3}
For positive integers $n, p$ and $r$, we have
\begin{align*}
\sum_{\ell=0}^n (\ell)^{(p)}h_{\ell}^{(r)}  h_{n-\ell}^{(s)}=A_{1}(p,r,s,n)S_{n}+B_{1}(p,r,s,n)H_{n}+C_{1}(p,r,s,n)\,,
\end{align*}
where
\begin{align*}
A_{1}(p,r,s,n)&=\sum_{m=0}^p s_{u}(p,m) A(m,r,s,n),\\
B_{1}(p,r,s,n)&=\sum_{m=0}^p s_{u}(p,m) B(m,r,s,n),\\
C_{1}(p,r,s,n)&=\sum_{m=0}^p s_{u}(p,m) C(m,r,s,n).\\
\end{align*}
\end{theorem}
\begin{proof}
\begin{align*}
&\quad \sum_{\ell=0}^n (\ell)^{(p)}h_{\ell}^{(r)}  h_{n-\ell}^{(s)}\\
&=\sum_{\ell=0}^n \sum_{m=0}^{p} s_{u}(p,m) {\ell}^{m}h_{\ell}^{(r)}  h_{n-\ell}^{(s)}\notag\\
&=\sum_{m=0}^{p} s_{u}(p,m) \sum_{\ell=0}^n {\ell}^{m}h_{\ell}^{(r)}  h_{n-\ell}^{(s)}\notag\\
&=\sum_{m=0}^{p} s_{u}(p,m) \biggl( A(m,r,s,n)S_{n}+B(m,r,s,n)H_{n}+C(m,r,s,n) \biggr) \notag\\
&=\left( \sum_{m=0}^p s_{u}(p,m) A(m,r,s,n) \right) S_{n}
+\left( \sum_{m=0}^p s_{u}(p,m) B(m,r,s,n) \right) H_{n}\\
&\quad +\left( \sum_{m=0}^p s_{u}(p,m) C(m,r,s,n)) \right)\,.
\end{align*}
\end{proof}

Before going further, we introduce some more notations.
\begin{align*}
&a_{4}(r,r,t):=\frac{1}{(r-1)!^{2}}\sum_{\substack{t_{1}+t_{2}=t\\ 0 \le t_{1},t_{2} \le r-1}} s_{u}(r,t_{1}+1) s_{u}(s,t_{2}+1)\,,\\
&a_{5}(r,r,t):=\frac{1}{(r-1)!^{2}}\sum_{\substack{t_{1}+t_{2}=t\\ 0 \le t_{1}\le r-1\\ 0 \le t_{2} \le r-2}} s_{u}(r,t_{1}+1) (t_{2}+1) s_{u}(s,t_{2}+2)\,.
\end{align*}

\begin{theorem}\label{hypersturc4}
For positive integers $n, p$ and $r$, we have
\begin{align*}
\sum_{\ell=0}^n \ell^{p}(h_{\ell}^{(r)})^{2}=A(2,p,r,n)(H_{n})^{2}+B(2,p,r,n)H_{n}+C(2,p,r,n)\,,
\end{align*}
where
\begin{align*}
&A(2,p,r,n)=\sum_{t=0}^{2r-2} a_{4}(r,r,t)A(p+t,n),\\
&B(2,p,r,n)=-\sum_{t=0}^{2r-2}a_{4}(r,r,t)F(p+t,n)-2H_{r-1}\sum_{t=0}^{2r-2} a_{4}(r,r,t)A(p+t,n)\\
&\qquad \qquad \quad +2\sum_{t=0}^{2r-3} a_{5}(r,r,t)A(p+t,n),\\
&C(2,p,r,n)=\sum_{t=0}^{2r-2} a_{4}(r,r,t)G(p+t,n)\\
&\qquad \qquad \quad +\sum_{\ell=0}^{n} {\ell}^{p} \binom{\ell+r-1}{r-1}^{2}\biggl(\frac{1}{\ell+1}+\cdots+\frac{1}{\ell+r-1}-H_{r-1}\biggr)^{2}\\
&\qquad \qquad \quad -2H_{r-1}\sum_{t=0}^{2r-2} a_{4}(r,r,t))\biggl(A(p+t,n)\frac{1}{n+1}-B(p+t,n)\biggr)\\
&\qquad \qquad \quad +2\sum_{t=0}^{2r-3} a_{5}(r,r,t)\biggl(A(p+t,n)\frac{1}{n+1}-B(p+t,n)\biggr).\\
\end{align*}
\end{theorem}

\begin{proof}
With the help of (\ref{h01}), we have
\begin{align*}
&\sum_{\ell=0}^{n} \ell^{p}(h_{\ell}^{(r)})^{2}\\
&=\sum_{\ell=0}^{n} {\ell}^{p} \binom{\ell+r-1}{r-1}^{2}(H_{\ell+r-1}-H_{r-1})^{2}\\
&=\sum_{\ell=0}^{n} {\ell}^{p} \binom{\ell+r-1}{r-1}^{2}\biggl(H_{\ell}^{2}+H_{r-1}^{2}-2H_{r-1}H_{\ell}+\biggl(\frac{1}{\ell+1}+\cdots+\frac{1}{\ell+r-1}\biggr)^{2}\\
&\quad +2H_{\ell}\biggl(\frac{1}{\ell+1}+\cdots+\frac{1}{\ell+r-1}\biggr)
-2H_{r-1}\biggl(\frac{1}{\ell+1}+\cdots+\frac{1}{\ell+r-1}\biggr)\biggr)\,.
\end{align*}
The above sum can be divided into $4$ parts.
\begin{align*}
&{\rm Part ~I}\\
&=\sum_{\ell=0}^{n} {\ell}^{p} \binom{\ell+r-1}{r-1}^{2}H_{\ell}^{2}\\
&=\sum_{\ell=0}^{n} {\ell}^{p} H_{\ell}^{2} \biggl(\frac{1}{(r-1)!}\sum_{k_{1}=1}^{r} s_{u}(r,k_{1}) \ell^{k_{1}-1}\biggr)^{2}\\
&=\sum_{\ell=0}^{n} {\ell}^{p} H_{\ell}^{2}\frac{1}{(r-1)!^{2}}\sum_{t=0}^{2r-2}\ell^t \sum_{\substack{t_{1}+t_{2}=t\\ 0 \le t_{1},t_{2} \le r-1}} s_{u}(r,t_{1}+1) s_{u}(r,t_{2}+1)\\
&=\sum_{t=0}^{2r-2} a_{4}(r,r,t) \sum_{\ell=0}^{n} {\ell}^{p+t} H_{\ell}^{2}\\
&=\sum_{t=0}^{2r-2} a_{4}(r,r,t)\biggl(A(p+t,n)(H_{n})^{2}-F(p+t,n)H_{n}+G(p+t,n)\biggr)\\
&=(H_{n})^{2}\sum_{t=0}^{2r-2} a_{4}(r,r,t)A(p+t,n)-H_{n}\sum_{t=0}^{2r-2}a_{4}(r,r,t)F(p+t,n)\\
&\quad +\sum_{t=0}^{2r-2} a_{4}(r,r,t)G(p+t,n)\,.
\end{align*}

$$
{\rm Part ~II}=\sum_{\ell=0}^{n} {\ell}^{p} \binom{\ell+r-1}{r-1}^{2}\biggl(\frac{1}{\ell+1}+\cdots+\frac{1}{\ell+r-1}-H_{r-1}\biggr)^{2} \,.
$$

\begin{align*}
&{\rm Part ~III}\\
&=-2\sum_{\ell=0}^{n} {\ell}^{p} \binom{\ell+r-1}{r-1}^{2}H_{\ell}H_{r-1}\\
&=-2H_{r-1}\sum_{\ell=0}^{n} {\ell}^{p} H_{\ell}\sum_{t=0}^{r+s-2}a_{4}(r,r,t)\ell^t\\
&=-2H_{r-1}\sum_{t=0}^{2r-2} a_{4}(r,r,t) \sum_{\ell=0}^{n} {\ell}^{p+t} H_{\ell}\\
&=-2H_{n}H_{r-1}\sum_{t=0}^{2r-2} a_{4}(r,r,t)A(p+t,n)\\
&\quad -2H_{r-1}\sum_{t=0}^{2r-2} a_{4}(r,r,t)\biggl(A(p+t,n)\frac{1}{n+1}-B(p+t,n)\biggr)\,.
\end{align*}

\begin{align*}
&{\rm Part ~IV}\\
&=2\sum_{\ell=0}^{n} {\ell}^{p} \binom{\ell+r-1}{r-1}^{2}H_{\ell}\biggl(\frac{1}{\ell+1}+\cdots+\frac{1}{\ell+r-1}\biggr)\\
&=2\sum_{\ell=0}^{n} {\ell}^{p} H_{\ell} \sum_{t=0}^{2r-3}a_{5}(r,r,t)\ell^t\\
&=2\sum_{t=0}^{2r-3} a_{5}(r,r,t) \sum_{\ell=0}^{n} {\ell}^{p+t} H_{n-\ell}\\
&=2\sum_{t=0}^{2r-3} a_{5}(r,r,t)(A(p+t,n)H_{n+1}-B(p+t,n))\\
&=2H_{n}\sum_{t=0}^{2r-3} a_{5}(r,r,t)A(p+t,n)\\
&\quad +2\sum_{t=0}^{2r-3} a_{5}(r,r,t)\biggl(A(p+t,n)\frac{1}{n+1}-B(p+t,n)\biggr)\,.
\end{align*}
Combining all these $4$ parts together, we get the desired result.
\end{proof}

\begin{theorem}\label{hypersturc5}
For positive integers $n, p$ and $r$, we have
\begin{align*}
\sum_{\ell=0}^n (\ell)^{(p)}(h_{\ell}^{(r)})^{2}=A_{2}(p,r,n)(H_{n})^{2}+B_{2}(p,r,n)H_{n}+C_{2}(p,r,n)\,,
\end{align*}
where
\begin{align*}
A_{2}(p,r,n)&=\sum_{m=0}^p s_{u}(p,m) A(2,m,r,n),\\
B_{2}(p,r,n)&=\sum_{m=0}^p s_{u}(p,m) B(2,m,r,n),\\
C_{2}(p,r,n)&=\sum_{m=0}^p s_{u}(p,m) C(2,m,r,n).\\
\end{align*}
\end{theorem}
\begin{proof}
\begin{align*}
&\sum_{\ell=0}^n (\ell)^{(p)}(h_{\ell}^{(r)})^{2}\\
&=\sum_{\ell=0}^n \sum_{m=0}^{p} s_{u}(p,m) {\ell}^{m}(h_{\ell}^{(r)})^{2}\notag\\
&=\sum_{m=0}^{p} s_{u}(p,m) \sum_{\ell=0}^n {\ell}^{m}(h_{\ell}^{(r)})^{2}\notag\\
&=\sum_{m=0}^{p} s_{u}(p,m) \biggl(A(2,m,r,n)(H_{n})^{2}+B(2,m,r,n)H_{n}+C(2,m,r,n)\biggr) \notag\\
&=\left( \sum_{m=0}^p s_{u}(p,m) A(2,m,r,n) \right) S_{n}
+\left( \sum_{m=0}^p s_{u}(p,m) B(2,m,r,n) \right) H_{n}\\
&\quad +\left( \sum_{m=0}^p s_{u}(p,m) C(2,m,r,n)) \right)\,.
\end{align*}
\end{proof}

\section{Summations involving generalized hyperharmonic numbers}

In this section, we consider some summations involving generalized hyperharmonic numbers.  We give expressions of the summations
$$
\sum_{\ell=0}^n \ell^{p}H_{\ell}^{(q,r)},\quad
\sum_{\ell=0}^n \ell^{p}H_{n-\ell}^{(r)},\quad
\sum_{\ell=0}^n \ell^{p}H_{n-\ell}^{(q,r)}\,.
$$

\begin{Lem}[\cite{Rusen}]\label{lemma3}
For $r, n, p \in \mathbb N$, we have
\begin{align*}
H_n^{(p,r)}=\sum_{m=0}^{r-1} \sum_{j=0}^{r-1-m} \hat{a}(r,m,j) n^{j} H_n^{(p-m)}\,.
\end{align*}
The coefficients $\hat{a}(r,m,j)$ satisfy the following recurrence relations:
\begin{align*}
&\hat{a}(r+1,r,0)=-\sum_{m=0}^{r-1} \hat{a}(r,m,r-m-1)\frac{1}{r-m}\,,\\
&\hat{a}(r+1,m,\ell)=\sum_{j=\ell-1}^{r-1-m} \frac{\hat{a}(r,m,j)}{j+1} \binom{j+1}{j-\ell+1}B_{j-\ell+1}^{+}\\
&\qquad\qquad(0\leq m \leq r-1, 1\leq \ell \leq r-m)\,,\\
&\hat{a}(r+1,m,0)=-\sum_{y=0}^{m} \sum_{j=max\{0, m-y-1\}}^{r-1-y}\hat{a}(r,y,j)D(r,m,j,y)\quad (0\leq m \leq r-1)\,,
\end{align*}
where
$$
D(r,m,j,y)=\sum_{\ell=max\{0, m-y-1\}}^{j} \frac{1}{j+1} \binom{j+1}{j-\ell}B_{j-\ell}^{+}\binom{\ell+1}{m-y}(-1)^{1+\ell-m+y}\,.
$$
The initial value is given by $\hat{a}(1,0,0)=1$.
\end{Lem}

\begin{theorem}\label{hypersturc6}
For positive integers $n, p, q$ and $r$ with $q \geq r$, we have
\begin{align*}
\sum_{\ell=0}^n \ell^{p}H_{\ell}^{(q,r)}=\sum_{y=0}^{p+r} b(p,r,n,y) H_{n}^{(q-y)} \,,
\end{align*}
where
\begin{align*}
&b(p,r,n,0)=\sum_{j=0}^{r-1} \hat{a}(r,0,j) A(p+j,n)\,,\notag\\
&b(p,r,n,y)=\sum_{j=0}^{r-1-y} \hat{a}(r,y,j) A(p+j,n)\\
&\qquad \qquad \quad -\sum_{\substack{m+t=y\\ 0 \le m\le r-1\\ 1 \le t \le p+r-m}} \sum_{j=max\{0,t-p-1\}}^{r-1-m} \hat{a}(r,m,j) H(p+j,t)\quad (1 \le y \le r-1)\,,\notag\\
&b(p,r,n,y)=-\sum_{\substack{m+t=y\\ 0 \le m\le r-1\\ 1 \le t \le p+r-m}} \sum_{j=max\{0,t-p-1\}}^{r-1-m} \hat{a}(r,m,j) H(p+j,t)\quad (r \le y \le p+r)\,.\notag\\
\end{align*}
\end{theorem}

\begin{proof}
With the help of Lemma \ref{lemma3}, we obtain
\begin{align*}
&\sum_{\ell=0}^n \ell^{p}H_{\ell}^{(q,r)}\\
&=\sum_{\ell=0}^n \ell^{p} \sum_{m=0}^{r-1} \sum_{j=0}^{r-1-m} \hat{a}(r,m,j) \ell^{j} H_{\ell}^{(q-m)}\notag\\
&=\sum_{m=0}^{r-1} \sum_{j=0}^{r-1-m} \hat{a}(r,m,j) \sum_{\ell=0}^n \ell^{p+j}H_{\ell}^{(q-m)}\notag\\
&=\sum_{m=0}^{r-1} \sum_{j=0}^{r-1-m} \hat{a}(r,m,j) \biggl(A(p+j,n) H_{n}^{(q-m)}-\sum_{t=1}^{p+j+1} H_{n}^{(q-m-t)} H(p+j,t)\biggr) \notag\\
&=\sum_{m=0}^{r-1} H_{n}^{(q-m)} \sum_{j=0}^{r-1-m} \hat{a}(r,m,j) A(p+j,n)\\
&\quad -\sum_{m=0}^{r-1} \sum_{t=1}^{p+r-m} H_{n}^{(q-m-t)} \sum_{j=max\{0,t-p-1\}}^{r-1-m}\hat{a}(r,m,j) H(p+j,t) \notag\\
&=\sum_{m=0}^{r-1} H_{n}^{(q-m)} \sum_{j=0}^{r-1-m} \hat{a}(r,m,j) A(p+j,n)\\
&\quad -\sum_{x=1}^{p+r}  H_{n}^{(q-x)}\sum_{\substack{m+t=x\\ 0 \le m\le r-1\\ 1 \le t \le p+r-m}} \sum_{j=max\{0,t-p-1\}}^{r-1-m} \hat{a}(r,m,j) H(p+j,t)\,.
\end{align*}
\end{proof}


\begin{Lem}\label{binoma6}
For $r, p, n \in \mathbb N$, we have
\begin{align*}
\sum_{k=1}^{n} \binom{k}{p} H_{n-k}^{(r)}=\sum_{t=0}^{p+1}a_{6}(p,n,t)H_{n-1}^{(r-t)}\,,
\end{align*}
where
\begin{align*}
&\quad a_{6}(p,n,t)=\frac{1}{(p+1)!}\sum_{i=t}^{p+1} s(p+1,i) \binom{i}{t} (-1)^{t} (n+1)^{i-t}\,.
\end{align*}
\end{Lem}

\begin{proof}
\begin{align*}
\sum_{k=1}^{n} \binom{k}{p} H_{n-k}^{(r)}
&=\sum_{k=1}^{n} \binom{k}{p} \sum_{j=1}^{n-k} \frac{1}{j^{r}}\\
&=\sum_{j=1}^{n-1} \frac{1}{j^{r}} \sum_{k=1}^{n-j}\binom{k}{p}\\
&=\sum_{j=1}^{n-1} \frac{1}{j^{r}} \binom{n-j+1}{p+1}\\
&=\sum_{j=1}^{n-1} \frac{1}{j^{r}} \frac{1}{(p+1)!}\sum_{i=0}^{p+1} s(p+1,i) (n-j+1)^{i}\\
&=\sum_{j=1}^{n-1} \frac{1}{j^{r}} \frac{1}{(p+1)!}\sum_{i=0}^{p+1} s(p+1,i) \sum_{t=0}^{i}\binom{i}{t} j^{t} (-1)^{t} (n+1)^{i-t}\\
&=\sum_{j=1}^{n-1} \frac{1}{j^{r}} \frac{1}{(p+1)!}\sum_{t=0}^{p+1} j^{t} \sum_{i=t}^{p+1} s(p+1,i) \binom{i}{t} (-1)^{t} (n+1)^{i-t}\\
&=\sum_{t=0}^{p+1} \sum_{j=1}^{n-1} \frac{1}{j^{r-t}} \frac{1}{(p+1)!} \sum_{i=t}^{p+1} s(p+1,i) \binom{i}{t} (-1)^{t} (n+1)^{i-t}\,.
\end{align*}
\end{proof}

\begin{theorem}\label{hypersturc7}
For positive integers $n, p$ and $r$, we have
\begin{align*}
\sum_{\ell=0}^n \ell^{p}H_{n-\ell}^{(r)}=\sum_{t=0}^{p+1} H_{n-1}^{(r-t)} \sum_{\ell=1}^{p} S(p,\ell){\ell!} a_{6}(\ell,n,t)\,.
\end{align*}
\end{theorem}

\begin{proof}
\begin{align*}
\sum_{\ell=0}^n \ell^{p}H_{n-\ell}^{(r)}
&=\sum_{\ell=0}^{p} S(p,\ell){\ell!} \sum_{k=0}^{n} \binom{k}{\ell} H_{n-k}^{(r)}\\
&=\sum_{\ell=1}^{p} S(p,\ell){\ell!} \sum_{k=1}^{n} \binom{k}{\ell} H_{n-k}^{(r)}\\
&=\sum_{\ell=1}^{p} S(p,\ell){\ell!} \sum_{t=0}^{p+1} a_{6}(\ell,n,t)H_{n-1}^{(r-t)}\\
&=\sum_{t=0}^{p+1} H_{n-1}^{(r-t)} \sum_{\ell=1}^{p} S(p,\ell){\ell!} a_{6}(\ell,n,t)\,.
\end{align*}
\end{proof}

Before going further, we introduce some more notations.
\begin{align*}
&a_{7}(p,n,t):=\sum_{\ell=1}^{p} S(p,\ell)\cdot {\ell!} \cdot a_{6}(\ell,n,t)\,,\\
&a_{8}(r,m,k,n):=\sum_{j=k}^{r-1-m}\hat{a}(r,m,j) \binom{j}{k}  (-1)^{k} n^{j-k}\,.
\end{align*}

\begin{theorem}\label{hypersturc8}
For positive integers $n, p, q$ and $r$, we have
$$
\sum_{\ell=0}^n \ell^{p}H_{n-\ell}^{(q,r)}=\sum_{y=0}^{p+r} c(p,r,n,y) H_{n-1}^{(q-y)} \,,
$$
where
$$
c(p,r,n,y)=\sum_{\substack{m+t=y\\ 0 \le m\le r-1\\ 0 \le t \le p+r-m}}\sum_{k=max\{0,t-p-1\}}^{r-1-m} a_{7}(p+k,n,t) a_{8}(r,m,k,n)\,.\\
$$
\end{theorem}

\begin{proof}
With the help of Lemma \ref{lemma3} and Theorem \ref{hypersturc7}, we obtain
\begin{align*}
&\quad \sum_{\ell=0}^n \ell^{p}H_{n-\ell}^{(q,r)}\\
&=\sum_{\ell=0}^n \ell^{p} \sum_{m=0}^{r-1} \sum_{j=0}^{r-1-m} \hat{a}(r,m,j) (n-\ell)^{j} H_{n-\ell}^{(q-m)}\\
&=\sum_{\ell=0}^n \ell^{p} \sum_{m=0}^{r-1} \sum_{j=0}^{r-1-m} \hat{a}(r,m,j) \sum_{k=0}^{j}\binom{j}{k} \ell^{k} (-1)^{k} n^{j-k} H_{n-\ell}^{(q-m)}\\
&=\sum_{\ell=0}^n \ell^{p} \sum_{m=0}^{r-1} H_{n-\ell}^{(q-m)} \sum_{k=0}^{r-1-m} \ell^{k} \sum_{j=k}^{r-1-m}\hat{a}(r,m,j) \binom{j}{k}  (-1)^{k} n^{j-k}\\
&=\sum_{m=0}^{r-1} \sum_{k=0}^{r-1-m} \sum_{\ell=0}^n \ell^{p+k} H_{n-\ell}^{(q-m)} a_{8}(r,m,k,n)\\
&=\sum_{m=0}^{r-1} \sum_{k=0}^{r-1-m} \sum_{t=0}^{p+k+1} H_{n-1}^{(q-m-t)} a_{7}(p+k,n,t) a_{8}(r,m,k,n)\\
&=\sum_{m=0}^{r-1} \sum_{t=0}^{p+r-m} H_{n-1}^{(q-m-t)}\sum_{k=max\{0,t-p-1\}}^{r-1-m} a_{7}(p+k,n,t) a_{8}(r,m,k,n)\\
&=\sum_{y=0}^{p+r} H_{n-1}^{(q-y)}\sum_{\substack{m+t=y\\ 0 \le m\le r-1\\ 0 \le t \le p+r-m}}\sum_{k=max\{0,t-p-1\}}^{r-1-m} a_{7}(p+k,n,t) a_{8}(r,m,k,n)\,.
\end{align*}
\end{proof}

\begin{Remark}
Summation formulas of type $\sum_{\ell=0}^{n} {(\ell)}^{(p)} H_{\ell}^{(q,r)}$,  $\sum_{\ell=0}^{n} {(\ell)}^{(p)} H_{n-\ell}^{(r)}$ and $\sum_{\ell=0}^{n} {(\ell)}^{(p)} H_{n-\ell}^{(q,r)}$ can be obtained in the same way with Theorem \ref{hypersturc3}.
\end{Remark}

\end{document}